\documentclass{tran-l}

\usepackage{amsthm}
\usepackage{amsmath}
\usepackage{amssymb}
\usepackage{latexsym}
\usepackage{enumerate}

\renewcommand{\thetheoremName}

\newtheorem{proposition[[]]}[theoremName]{Proposition G}

\newtheorem{theorem}{Theorem}[section]
\newtheorem{lemma}[theorem]{Lemma}

\newtheorem{proposition}[theorem]{Proposition}
\newtheorem{corollary}[theorem]{Corollary}

\theoremstyle{definition}
\newtheorem{definition}[theorem]{Definition}

\newtheorem{remark}{Remark}

\numberwithin{equation}{section}

\newcommand{\nablao}{\mbox{$\overline{\nabla}$}}
\newcommand{\Deltao}{\mbox{$\bar{\Delta}$}}
\newcommand{\R}[1]{\mbox{${\mathbb R}^{#1}$}}

\newcommand{\g}[2]{\mbox{$\langle #1 ,#2 \rangle$}}
\newcommand{\m}{\mbox{$\Sigma$}}

\def\beq{\begin{equation}}
\def\eeq{\end{equation}}

\begin{document}

\title[Geometric Analysis of Lorentzian distance]
{GEOMETRIC ANALYSIS OF LORENTZIAN DISTANCE FUNCTION ON SPACELIKE HYPERSURFACES}
\author[L. J. Al\'\i as]{Luis J. Al\'\i as$^{\#}$}
\address{Departamento de Matem\'{a}ticas, Universidad de Murcia, E-30100 Espinardo, Murcia, Spain}
\email{ljalias@um.es}
\author[A. Hurtado]{Ana Hurtado*}
\address{Departamento de Geometr\'\i a y Topolog\'\i a, Universidad de Granada, E-18071 Granada,
Spain.}
 \email{ahurtado@ugr.es}
\author[V. Palmer]{Vicente Palmer*}
\address{Departament de Matem\`{a}tiques, Universitat Jaume I, E-12071 Castell\'{o},
Spain.}
 \email{palmer@mat.uji.es}
\thanks{$^{\#}$ This research is a result of the activity developed within the framework of the
Programme in Support of Excellence Groups of the Regi\'{o}n de Murcia, Spain, by Fundaci\'{o}n S\'{e}neca,
Regional Agency for Science and Technology (Regional Plan for Science and Technology 2007-2010).
Research partially supported by MEC project MTM2007-64504, and Fundaci\'{o}n S\'{e}neca project 04540/GERM/06,
Spain.\\
\indent * Supported by Spanish MEC-DGI grant No.MTM2007-62344 and the Bancaixa-Caixa Castell\'{o}
Foundation}


%


\begin{abstract}
Some analysis on the Lorentzian distance in a spacetime with controlled sectional
(or Ricci) curvatures is done. In particular, we focus on the study of the
restriction of such distance to a spacelike hypersurface
satisfying the Omori-Yau maximum principle.
As a consequence, and under appropriate hypotheses on the
(sectional or Ricci) curvatures of the ambient spacetime, we obtain sharp estimates
for the mean curvature of those hypersurfaces. Moreover, we also give a suficient
condition for its hyperbolicity.
\end{abstract}

\maketitle

\section{Introduction}

Let $M^{n+1}$ be a $(n+1)$-dimensional spacetime, and consider either
$d_p$, the Lorentzian distance from a fixed point $p \in M$, or $d_N$, the Lorentzian distance from a fixed achronal
spacelike hypersurface $N$. Under
suitable conditions those Lorentzian distances are
differentiable at least in a \lq\lq sufficiently near
chronological future" of the point $p$ or of the hypersurface $N$, so that some classical analysis
can be done on those functions.

In this setting, over the past 25 years comparison theory and geometric analysis of the distance function has been
effectively extended
and applied to Lorentzian manifolds. In particular, it played an
important role in the proof of the Lorentzian splitting theorem, the
spacetime analogue of the Cheeger-Gromoll splitting theorem, first established by Galloway \cite{Ga1} and by
Beem, Ehrlich, Markvorsen and Galloway \cite{BEMG}, and subsequently improved by Eschenburg \cite{Es},
Galloway \cite{Ga2}, and Newman \cite{Ne}. In those works, one needs to understand the geometry, i.e., mean curvature,
of the spacelike level sets of the Lorentzian distance function from a fixed point. As in the Riemannian case, this is
analytically expressed in terms of the (Lorentzian)  Laplacian (called also d'Alembertian, in the Lorentzian case) of
the distance function. More recently, in the paper \cite{EGK}, Erkekoglu, Garc\'\i a-R\'\i o and Kupeli obtained Hessian and Laplacian comparison theorems for those
Lorentzian distance functions from comparisons of the sectional curvatures of the Lorentzian manifold, following the
lines of Greene and Wu in their classical book \cite{GreW}, where it were obtained  the same comparison for the Hessian
and the Laplacian of the Riemannian distance function from estimates of sectional curvatures.

In this paper we shall study the Lorentzian distance function
restricted to a spacelike hypersurface $\Sigma^n$ immersed into $M^{n+1}$. In particular, we shall consider spacelike
hypersufaces whose image under the immersion is bounded in the ambient
spacetime, in the sense that the Lorentzian distance either from a fixed
point or from $N$ to the hypersurface is bounded from above.

Inspired by the works \cite{AA1}, \cite{AA2}  and \cite{Wu}, we derive sharp estimates for the mean curvature of such
hypersurfaces, provided that either {\bf (i)} the Ricci curvature of the ambient spacetime $M^{n+1}$ is bounded from below on
timelike directions (Theorem \ref{mainTh1UpperB} and Theorem \ref{mainTh1UpperBachro}), which obviously includes the case where the sectional curvatures of
all timelike planes of $M^{n+1}$ are bounded from above, or {\bf (ii)} the sectional curvatures of all timelike planes of
$M^{n+1}$ are bounded from below (Theorem \ref{mainTh1LowerB} and Theorem \ref{mainTh1LowerBachro}), or {\bf (iii)} the sectional curvature of $M^{n+1}$
is constant (Theorem \ref{mainTh2UpperLowB}), widely extending previous results in the previous papers. In particular, we
establish a Bernstein-type result for the Lorentzian distance, (see Corollary \ref{corlevelsets}), which improves
Theorem 1 in \cite{AA1} (see Remark \ref{remark1} and Corollary \ref{coro3}) and extends it to arbitrary Lorentzian space
forms.

On the other hand, we also study some function theoretic properties on mean-curvature-controlled spacelike hypersurfaces,
via the control of the Hessian of the Lorentzian distance, following the lines in \cite{MP1} and \cite{MP2}. In particular,
we show that spacelike hypersurfaces with mean curvature bounded from above are hyperbolic, in the sense that they admit a
non-constant positive superharmonic function, when the ambient spacetime has timelike sectional curvatures bounded from
below (see Theorem \ref{MainTh4Hyperbolicity} and Theorem \ref{MainTh4Hyperbolicityachro}).

\subsection{Outline of the paper}
We devote Section 2  and Section 3 to presenting the basic concepts involved and establishing our comparison
analysis of the Hessian of the Lorentzian distance function from a point, respectively, together with the basic comparison
inequalities for the Laplacian.
In Section 4 we state and prove the sharp estimates for the mean curvature of spacelike hypersurfaces bounded by a level set of the Lorentzian distance function from a point. In Section 5 we extend
our geometric analysis to the Lorentzian distance function from an achronal spacelike hypersurface, establishing the corresponding
results for that function. Finally the proofs of hyperbolicity are presented in Section 6.

\subsection{Acknowledgements}
This work has been partially done during the stay of the third named author at the Department of Mathematics of Universidad de Murcia and  the Max Planck Institut für Mathematik in Bonn,  where he enjoyed part of a sabbatical
leave. He would like to thank the staff at these institutions for the cordial hospitality during this period. The
authors also thank to the referee for valuable suggestions which improved the paper.

\section{Preliminaries}
Consider $M^{n+1}$ an $(n+1)$-dimensional spacetime, that is, a time-oriented Lo\-ren\-tzian manifold of dimension
$n+1\geq 2$. Let $p,q$ be points in $M$. Using the standard terminology and notation from Lorentzian geometry, one says
that $q$ is in the chronological future of $p$, written $p\ll q$, if there exists a future-directed timelike curve from
$p$ to $q$. Similarly, $q$ is in the causal future of $p$, written $p<q$, if there exists a future-directed causal
(i.e., nonspacelike) curve from $p$ to $q$. Obviously, $p\ll q$ implies $p<q$. As usual, $p\leq q$ means that either $p<q$ or $p=q$.

For a subset $S\subset M$, one defines the chronological future of $S$ as
\[
I^+(S)=\{ q\in M : p\ll q \mbox{ for some } p\in S\},
\]
and the causal future of $S$ as
\[
J^+(S)=\{ q\in M : p\leq q \mbox{ for some } p\in S\}.
\]
Thus $S\cup I^+(S)\subset J^+(S)$.

In particular, the chronological future $I^+(p)$ and the causal future $J^+(p)$ of a point $p\in M$ are
\[
I^+(p)=\{ q\in M : p\ll q \}, \quad \mbox{and} \quad J^+(p)=\{ q\in M : p\leq q \}.
\]
As is well-known, $I^+(p)$ is always open, but $J^+(p)$ is neither open nor closed in general.

If $q\in J^+(p)$, then the Lorentzian distance $d(p,q)$ is the supremum of the Lorentzian lengths of all
the future-directed causal curves from $p$ to $q$ (possibly, $d(p,q)=+\infty$).
If $q\notin J^+(p)$, then the Lorentzian distance $d(p,q)=0$ by definition.
Specially, $d(p,q)>0$ if and only if $q\in I^+(p)$.

The Lorentzian distance function $d:M\times M\rightarrow [0,+\infty]$ for an arbitrary
spacetime may fail to be continuous in general, and may also fail to be finite valued. As a matter of fact, globally
hyperbolic spacetimes turn out to be the natural class of spacetimes for which the Lorentzian distance function is
finite-valued and continuous.

Given a point $p\in M$, one can define the Lorentzian distance function $d_p:M\rightarrow [0,+\infty]$ with
respect to $p$ by
\[
d_p(q)=d(p,q).
\]
In order to guarantee the smoothness of $d_p$, we need to restrict this function on certain special
subsets of $M$. Let $T_{-1}M|_{p}$ be the fiber of the unit future observer bundle of $M$ at $p$, that is,
\[
T_{-1}M|_{p}=\{v\in T_pM : v \mbox{ is a future-directed timelike unit vector} \}.
\]
Define the function $s_p:T_{-1}M|_{p}\rightarrow [0,+\infty]$ by
\[
s_p(v)=\sup\{ t\geq 0 : d_p(\gamma_v(t))=t\},
\]
where $\gamma_v:[0,a)\rightarrow M$ is the future inextendible geodesic starting at $p$ with initial velocity $v$.
Then, one can define
\[
\tilde{\mathcal{I}}^{+}(p)=\{ tv: \mbox{ for all } v\in T_{-1}M|_{p} \mbox{ and } 0<t<s_p(v) \}
\]
and consider the subset $\mathcal{I}^{+}(p)\subset M$ given by
\[
\mathcal{I}^{+}(p)=\mathrm{exp}_p(\mathrm{int}(\tilde{\mathcal{I}}^{+}(p)))\subset I^{+}(p).
\]
Observe that
\[
\mathrm{exp}_p: \mathrm{int}(\tilde{\mathcal{I}}^{+}(p))\rightarrow \mathcal{I}^{+}(p)
\]
is a diffeomorphism and $\mathcal{I}^{+}(p)$ is an open subset (possible empty).

For instance, when $c\geq 0$, the
Lorentzian space form $M^{n+1}_c$ is globally hyperbolic and geodesically complete, and every future directed
timelike unit geodesic $\gamma_c$ in $M^{n+1}_c$ realizes the Lorentzian distance between its points. In particular, if
$c\geq 0$ then $\mathcal{I}^+(p)=I^+(p)$ for every point $p\in M^{n+1}_c$ (see \cite[Remark 3.2]{EGK}). However,
when $c<0$ it can be easily seen that $\mathcal{I}^+(p)=\emptyset$ for every point $p\in\mathbb{H}^{n+1}_1$, where
$\mathbb{H}^{n+1}_1$ is the anti-de-Sitter space, that is, the standard model of a simply connected Lorentzian space form
with negative curvature. In fact, at each point $p\in\mathbb{H}^{n+1}_1$, it holds that every future directed timelike geodesic in
$\mathbb{H}^{n+1}_1$ starting at $p$ is closed, which implies that $d(p,\gamma(t))=+\infty$
for every $t\in\R{}$. The following result summarizes the main properties about the Lorentzian distance function (see
\cite[Section 3.1]{EGK}).

\begin{lemma}
\label{lemaEduardo1}
Let $M$ be a spacetime and $p\in M$.
\begin{enumerate}
\item If $M$ is strongly causal at $p$, then $s_p(v)>0$ for all $ v\in T_{-1}M|_{p}$ and
$\mathcal{I}^{+}(p)\neq\emptyset$.
\item If $\mathcal{I}^{+}(p)\neq\emptyset$, then the Lorentzian distance function $d_p$ is smooth on
$\mathcal{I}^{+}(p)$ and its gradient $\nablao d_p$ is a past-directed timelike (geodesic) unit vector field on $\mathcal{I}^{+}(p)$.
\end{enumerate}
\end{lemma}

\section{Analysis of the Lorentzian distance function from a point}
\label{comp}
This section has two parts: in the first one, we are going to present estimates for the Hessian of the Lorentzian distance from a point in a Lorentzian manifold in terms of bounds for its timelike sectional curvatures.
In the second part, we obtain estimates for the Hessian and the Laplacian of the Lorentzian distance from a point restricted to a spacelike hypersurface, based in the previous comparisons.

For every $c\in\R{}$, let us define
\[
f_c(s)=
\begin{cases}
\sqrt{c}\coth(\sqrt{c}\,s) &\text{if $c>0$ and $s>0$}\\
\phantom{\sqrt{c}} 1/s &\text{if $c=0$ and $s>0$}\\
\sqrt{-c}\cot(\sqrt{-c}\,s) &\text{if $c<0$ and $0<s<\pi/\sqrt{-c}$}.
\end{cases}
\]
The function $f_c$ arises naturally when computing the index form of a timelike geodesic in a Lorentzian space form of constant curvature $c$, $M^{n+1}_c$. Indeed,
let $\gamma_c:[0,s]\rightarrow M^{n+1}_c$ be a future directed timelike unit geodesic (with $s<\pi/\sqrt{-c}$ when $c<0$), and let $J_c$ be
a Jacobi field along $\gamma_c$ such that $J_c(0)=0$ and $J_c(s)=x\perp\gamma'_c(s)$. Using the Jacobi equation along $\gamma_c$, it is straightforward to see that $J_c(t)$ is given by
$J_c(t)=\mathrm{s}_c(t)Y_c(t)$, where
\beq
\label{sc}
\mathrm{s}_c(t)=
\begin{cases}
\frac{\sinh(\sqrt{c}\,t)}{\sinh(\sqrt{c}\,s)} & \text{if $c>0$ and $0\leq t\leq s$}\\
\phantom{\sqrt{c}} t/s  & \text{if $c=0$ and $0\leq t\leq s$}\\
\frac{\sin(\sqrt{-c}\,t)}{\sin(\sqrt{-c}\,s)} &\text{if $c<0$ and $0\leq t\leq s<\pi/\sqrt{-c}$},
\end{cases}
\eeq
and $Y_c(t)$ is the parallel vector field along $\gamma_c$ such that $Y_c(s)=x$ (and hence, $\g{Y_c(t)}{Y_c(t)}_c=\g{x}{x}$ for every $t$). Thus,
\[
\g{J_c(t)}{J_c(t)}_c=\mathrm{s}_c(t)^2\g{x}{x} \quad \mathrm{and} \quad
\g{J'_c(t)}{J'_c(t)}_c=\mathrm{s}'_c(t)^2\g{x}{x},
\]
and we can compute explicitly the index form of $\gamma_c$ on $J_c$ by
\begin{eqnarray}
\label{eq3}
I_{\gamma_c}(J_c,J_c) & = & -\int_0^s\left(\g{J'_c(t)}{J'_c(t)}_c+c\g{J_c(t)}{J_c(t)}_c\right)dt\\
\nonumber {} & = & -\int_0^s\left(\mathrm{s}'_c(t)^2+c\mathrm{s}_c(t)^2\right)dt\ \g{x}{x}=-f_c(s)\g{x}{x}.
\end{eqnarray}

On the other hand, it is worth pointing out that $f_c(s)$ is the future mean curvature of the Lorentzian sphere of radius $s$ in
the Lorentzian space form $M^{n+1}_c$ (when $\mathcal{I}^{+}(p)\neq\emptyset$), that is, the level set
\[
\Sigma_c(s)=\{ q \in\mathcal{I}^{+}(p) : d_{p}(q)=s \}\subset M^{n+1}_c.
\]
To see this note that the future-directed timelike unit normal field globally defined on $\Sigma_c(s)$ is the gradient
$-\nablao d_p$

Our first result assumes that the sectional curvatures of the timelike planes of $M$ are bounded from above by a constant
$c$.
\begin{lemma}
\label{comparacion1}
Let $M^{n+1}$ be an $(n+1)$-dimensional spacetime such that $K_M(\Pi)\leq c$, $c\in\R{}$,
for all timelike planes in $M$. Assume that there exists a point $p\in M$ such that
$\mathcal{I}^{+}(p)\neq\emptyset$, and let $q\in\mathcal{I}^{+}(p)$,
(with  $d_p(q)<\pi/\sqrt{-c}$ when $c<0$). Then for every spacelike vector $x\in T_{q}M$ orthogonal to $\nablao d_p(q)$ it
holds that
\begin{equation}
\label{eq31a}
\nablao^2d_p(x,x)\geq -f_c(d_p(q))\g{x}{x},
\end{equation}
where $\nablao^2$ stands for the Hessian operator on $M$. When $c<0$ but $d_p(q)\geq\pi/\sqrt{-c}$, then it still holds
that
\begin{equation}
\label{eq31b}
\nablao^2d_p(x,x)\geq -\frac{1}{d_p(q)}\g{x}{x}\geq-\frac{\sqrt{-c}}{\pi}\g{x}{x}.
\end{equation}
\end{lemma}
\begin{proof}
The proof follows the ideas of the proof of \cite[Theorem 3.1]{EGK}. Let
$v=\mathrm{exp}_p^{-1}(q)\in\mathrm{int}(\tilde{\mathcal{I}}^+(p))$ and let
$\gamma(t)=\mathrm{exp}_p(tv)$, $0\leq t<s_p(v)$, the radial future directed unit timelike geodesic with $\gamma(0)=p$
and $\gamma(s)=q$, where $s=d_p(q)$. Recall that $\gamma'(s)=-\nablao d_p(q)$, (see \cite[Proposition 3.2]{EGK}).
From \cite[Proposition 3.3]{EGK}, we know that
\[
\nablao^2d_p(x,x)=-\int_0^s(\g{J'(t)}{J'(t)}-\g{R(J(t),\gamma'(t))\gamma'(t)}{J(t)})dt=I_\gamma(J,J)
\]
where $J$ is the (unique) Jacobi field along $\gamma$ such that $J(0)=0$ and $J(s)=x$. Since
$\gamma:[0,s]\rightarrow\mathcal{I}^+(p)$ and
$\mathrm{exp}_p: \mathrm{int}(\tilde{\mathcal{I}}^{+}(p))\rightarrow \mathcal{I}^{+}(p)$ is a diffeomorphism, then there
is no conjugate point of $\gamma(0)$ along the geodesic $\gamma$. Therefore, by the maximality of the index of Jacobi
fields \cite[Theorem 10.23]{BEE} we get that
\begin{equation}
\label{eq4}
\nablao^2d_p(x,x)=I_\gamma(J,J)\geq I_{\gamma}(X,X).
\end{equation}
for every vector field $X$ along $\gamma$ such that $X(0)=J(0)=0$, $X(s)=J(s)=x$ and $X(t)\perp\gamma'(t)$ for every $t$.
Observe that, for all these vector fields $X$,
\begin{eqnarray*}
I_{\gamma}(X,X) & = & -\int_0^s(\g{X'(t)}{X'(t)}-\g{R(X(t),\gamma'(t))\gamma'(t)}{X(t)})dt\\
{} & = & -\int_0^s(\g{X'(t)}{X'(t)}+K(t)\g{X(t)}{X(t)})dt,
\end{eqnarray*}
where $K(t)$ stands for the sectional curvature of the timelike plane spanned by $X(t)$ and $\gamma'(t)$. Thus,
$K(t)\leq c$, and from (\ref{eq4}) we obtain that
\begin{equation}
\label{eq5}
\nablao^2d_p(x,x)\geq -\int_0^s(\g{X'(t)}{X'(t)}+c\g{X(t)}{X(t)})dt,
\end{equation}

Assume now that $s=d_p(q)<\pi/\sqrt{-c}$ if $c<0$, and let $Y(t)$ be the (unique) parallel vector field along $\gamma$ such
that $Y(s)=x$. Then, we may define $X(t)=\mathrm{s}_c(t)Y(t)$, where $\mathrm{s}_c(t)$ is the function given by (\ref{sc}).
Observe that $X$ is orthogonal to $\gamma$ and $X(0)=0$ and $X(s)=x$. Moreover,
\[
\g{X(t)}{X(t)}=\mathrm{s}_c(t)^2\g{x}{x} \quad \mathrm{and} \quad
\g{X'(t)}{X'(t)}=\mathrm{s}'_c(t)^2\g{x}{x}.
\]
Therefore, using $X$ in (\ref{eq5}) we get that
\[
\nablao^2d_p(x,x)\geq
-\int_0^s\left(\mathrm{s}'_c(t)^2+c\mathrm{s}_c(t)^2\right)dt\ \g{x}{x}=-f_c(s)\g{x}{x}.
\]
This finishes the proof of \ref{eq31a}.
Finally, when $c<0$ but $d_p(q)\geq\pi/\sqrt{-c}$, then $K_M(\Pi)\leq c<0$ and we may apply our estimate (\ref{eq31a})
for the constant $c=0$, so that
\[
\nablao^2d_p(x,x)\geq-f_0(d_p(q))\g{x}{x}=-\frac{1}{d_p(q)}\g{x}{x}\geq-\frac{\sqrt{-c}}{\pi}\g{x}{x}.
\]
\end{proof}

On the other hand, under the assumption that the sectional curvatures of the timelike planes of $M$ are bounded from
below by a constant $c$, we get the following result.
\begin{lemma}
\label{comparacionCHINO}
Let $M^{n+1}$ be an $(n+1)$-dimensional spacetime such that $K_M(\Pi)\geq c$, $c\in\R{}$,
for all timelike planes in $M$. Assume that there exists a point $p\in M$ such that
$\mathcal{I}^{+}(p)\neq\emptyset$, and let $q\in\mathcal{I}^{+}(p)$ (with $d_p(q)<\pi/\sqrt{-c}$ when $c<0$).
Then, for every spacelike vector $x\in T_{q}M$ orthogonal to $\nablao d_p(q)$ it holds that
\[
\nablao^2d_p(x,x)\leq -f_c(d_p(q))\g{x}{x},
\]
where $\nablao^2$ stands for the Hessian operator on $M$.
\end{lemma}
\begin{proof}
Similarly, the proof follows the ideas of the proof of \cite[Theorem 3.1]{EGK} (see also \cite[Lemma 8]{Wu}).
As in the previous proof, let $\gamma:[0,s]\rightarrow\mathcal{I}^+(p)$ be the radial future directed unit timelike
geodesic with $\gamma(0)=p$ and $\gamma(s)=q$, where $s=d_p(q)$.
From \cite[Proposition 3.3]{EGK}, we know that
\begin{eqnarray*}
\nablao^2d_p(x,x) & = & -\int_0^s(\g{J'(t)}{J'(t)}-\g{R(J(t),\gamma'(t))\gamma'(t)}{J(t)})dt\\
{} & = & -\int_0^s(\g{J'(t)}{J'(t)}+K(t)\g{J(t)}{J(t)})dt,
\end{eqnarray*}
where $J$ is the (unique) Jacobi field along $\gamma$ such that $J(0)=0$ and $J(s)=x$, and $K(t)$ stands for
the sectional curvature of the timelike plane spanned by $J(t)$ and $\gamma'(t)$. Thus, $K(t)\geq c$ and hence
\beq
\label{eq1b}
\nablao^2d_p(x,x)\leq -\int_0^s(\g{J'(t)}{J'(t)}+c\g{J(t)}{J(t)})dt.
\eeq

Let $\{E_1(t),\ldots,E_{n+1}(t)\}$ be an orthonormal frame of parallel vector fields along $\gamma$ such that
$E_{n+1}=\gamma'$. Write $J(t)=\sum_{i=1}^n\lambda_i(t)E_i(t)$, so that $J'(t)=\sum_{i=1}^n\lambda'_i(t)E_i(t)$.
Consider $\gamma_c:[0,s]\rightarrow M^{n+1}_c$ a future directed timelike unit geodesic in the Lorentzian
space form of constant curvature $c$, and let $\{E^c_1(t),\ldots,E^c_{n+1}(t)\}$ be an orthonormal frame of parallel
vector fields along $\gamma_c$ such that $E^c_{n+1}=\gamma'_c$. Define $X_c(t)=\sum_{i=1}^n\lambda_i(t)E^c_i(t)$, and
observe that
\begin{eqnarray*}
\g{J'(t)}{J'(t)}+c\g{J(t)}{J(t)} & = & \sum_{i=1}^n\left(\lambda'_i(t)^2+c\lambda_i(t)^2\right) \\
{} & = & \g{X'_c}{X'_c}_c+c\g{X_c}{X_c}_c \\
{} & = & \g{X'_c}{X'_c}_c-\g{R_c(X_c,\gamma'_c)\gamma'_c}{X_c}_c,
\end{eqnarray*}
where $\g{}{}_c$ and $R_c$ stand for the metric and Riemannian tensors of $M^{n+1}_c$. Then, (\ref{eq1b}) becomes
\begin{equation}
\label{eq1}
\nablao^2d_p(x,x)\leq I_{\gamma_c}(X_c,X_c),
\end{equation}
where $I_{\gamma_c}$ is the index form of $\gamma_c$ in the Lorentzian space form $M^{n+1}_c$.

Since there are no conjugate points of $\gamma_c(0)$ along
$\gamma_c$ (recall that $s<\pi/\sqrt{-c}$ when $c<0$), by the maximality of the index of Jacobi fields and equation (\ref{eq3}), we know that
\begin{equation}
\label{eq2}
I_{\gamma_c}(X_c,X_c)\leq I_{\gamma_c}(J_c,J_c)=-f_c(s)\g{x}{x},
\end{equation}
where $J_c$ stands for the Jacobi field along $\gamma_c$ such that $J_c(0)=X_c(0)=0$ and $J_c(s)=X_c(s)\perp\gamma'_c(s)$.
 The result directly follows from here and (\ref{eq1}).
\end{proof}

Observe that if $K_M(\Pi)\leq c$ for all timelike planes in $M$ (curvature hypothesis in Lemma \ref{comparacion1}), then
for every unit timelike vector $Z\in TM$
\[
\mathrm{Ric}_M(Z,Z)=-\sum_{i=1}^nK_M(E_i\wedge Z)\geq -nc,
\]
where $\{E_1,\ldots,E_n, E_{n+1}=Z\}$ is a local orthonormal frame. Our next result holds under this weaker
hypothesis on the Ricci curvature of $M$. When $c=0$ this is nothing but the so called timelike convergence condition.
\begin{lemma}
\label{comparacion3}
Let $M^{n+1}$ be an $(n+1)$-dimensional spacetime such that
\[
\mathrm{Ric}_M(Z,Z)\geq -nc, \quad c\in\R{},
\]
for every unit timelike vector $Z$. Assume that there exists a point $p\in M$ such that
$\mathcal{I}^{+}(p)\neq\emptyset$, and let $q\in\mathcal{I}^{+}(p)$,
(with  $d_p(q)<\pi/\sqrt{-c}$ when $c<0$). Then
\begin{equation}
\label{eq33a}
\Deltao d_p(q)\geq -nf_c(d_p(q)),
\end{equation}
where $\Deltao$ stands for the (Lorentzian) Laplacian operator on $M$. When $c<0$ but $d_p(q)\geq\pi/\sqrt{-c}$, then it
still holds that
\begin{equation}
\label{eq33b}
\Deltao d_p(q)\geq-\frac{n}{d_p(q)}\geq-\frac{n\sqrt{-c}}{\pi}.
\end{equation}
\end{lemma}
\begin{proof}
The proof follows the ideas of the proof of \cite[Lemma 3.1]{EGK}. Observe that our criterion here for the definition
of the Laplacian operator is the one in \cite{O'N} and \cite{BEE}, that is, $\Deltao=\mathrm{tr}(\nablao^2)$.
Let $v=\mathrm{exp}_p^{-1}(q)\in\mathrm{int}(\tilde{\mathcal{I}}^+(p))$ and let
$\gamma(t)=\mathrm{exp}_p(tv)$, $0\leq t<s_p(v)$, the radial future directed unit timelike geodesic with $\gamma(0)=p$
and $\gamma(s)=q$, where $s=d_p(q)$. Let $\{e_1,\ldots,e_{n}\}$ be orthonormal vectors in $T_{q}M$ orthogonal to
$\gamma'(s)=-\nablao d_p(q)$, so that
\beq
\label{eq5bbis}
\Deltao d_p(q)=\sum_{j=1}^n\nablao^2 d_p(e_j,e_j).
\eeq
As in the proof of Lemma \ref{comparacion1}, we have that, for every $j=0,\ldots,n$,
\[
\nablao^2d_p(e_j,e_j)\geq I_\gamma(X_j,X_j)
\]
for every vector field $X_j$ along $\gamma$ such that $X_j(0)=0$, $X_j(s)=e_j$ and $X_j(t)\perp\gamma'(t)$ for every $t$,
which by (\ref{eq5bbis}) implies that
\beq
\label{eq5bis}
\Deltao d_p(q)\geq\sum_{j=1}^n I_\gamma(X_j,X_j).
\eeq

Assume now that $s=d_p(q)<\pi/\sqrt{-c}$ when $c<0$, and let $\{E_1(t),\ldots,E_{n+1}(t)\}$ be an orthonormal frame of parallel vector fields along $\gamma$ such that
$E_j(s)=e_j$ for every $j=0,\ldots,n$, and $E_{n+1}=\gamma'$.

Define
\[
X_j(t)=\mathrm{s}_c(t)E_j(t), \quad j=1,\ldots,n,
\]
where $\mathrm{s}_c(t)$ is the function given by (\ref{sc}).
Since $X_j$ is orthogonal to $\gamma$ and $X_j(0)=0$ and $X_j(s)=e_j$, we may use $X_j$ in (\ref{eq5bis}). Observe
that $\{X_1,\ldots,X_n\}$ are orthogonal along $\gamma$, and
\[
\g{X_j(t)}{X_j(t)}=\mathrm{s}_c(t)^2 \quad \mathrm{and} \quad \g{X_j'(t)}{X_j'(t)}=\mathrm{s}'_c(t)^2,
\]
for every $j=0,\ldots,n$. Therefore, for every $j$ we get
\[
I_\gamma(X_j,X_j)=-\int_0^s(\mathrm{s}'_c(t)^2-\mathrm{s}_c(t)^2\g{R(E_j(t),\gamma'(t))\gamma'(t)}{E_j(t)})dt,
\]
and then
\begin{eqnarray*}
\sum_{j=1}^n I_\gamma(X_j,X_j) & = & -n\int_0^s
\left(\mathrm{s}'_c(t)^2-\frac{\mathrm{s}_c(t)^2}{n}\mathrm{Ric}_M(\gamma'(t),\gamma'(t))\right)dt\\
{} & \geq & -n\int_0^s\left(\mathrm{s}'_c(t)^2+c\mathrm{s}_c(t)^2\right)dt=-nf_c(s).
\end{eqnarray*}
Thus, by (\ref{eq5bis}) we get (\ref{eq33a}).
Finally, when $c<0$ but $d_p(q)\geq\pi/\sqrt{-c}$, then $\mathrm{Ric}_M(Z,Z)\geq -nc>0$ and we may apply (\ref{eq33a})
for the constant $c=0$, which yields
\[
\Deltao d_p(q)\geq -nf_0(d_p(q))=-\frac{n}{d_p(q)}\geq-\frac{n\sqrt{-c}}{\pi}.
\]
\end{proof}

Now we are ready to start our analysis of the Lorentzian distance function from a point on a spacelike hypersurface in $M$. Let $\psi:\Sigma^n\rightarrow M^{n+1}$ be a spacelike hypersurface immersed into the spacetime $M$. Since $M$
is time-oriented, there exists a unique future-directed timelike unit normal field $\nu$ globally defined on \m. We will refer to $\nu$ as the future-directed Gauss map of \m. Let $A$ stand for the shape operator of \m\ with respect
to $\nu$. The $H=-(1/n)\mathrm{tr}(A)$ defines the future mean curvature of \m. The choice of the sign $-$ in our definition of $H$ is motivated by the fact that in that case the mean curvature vector is given by
$\overrightarrow{H}=H\nu$. Therefore, $H(p)>0$ at a point $p\in\m$ if and only if $\overrightarrow{H}(p)$ is future-directed.

Let us assume that there exists a point $p\in M$ such that $\mathcal{I}^{+}(p)\neq\emptyset$ and that $\psi(\m)\subset\mathcal{I}^{+}(p)$. Let $r=d_p$ denote the Lorentzian distance function from $p$, and let
$u=r\circ\psi:\m\rightarrow(0,\infty)$ be the function $r$ along the hypersurface, which is a smooth function on \m.

Our first objective is to compute the Hessian of $u$ on \m. To do that, observe that
\[
\nablao r=\nabla u-\g{\nablao r}{\nu}\nu
\]
along \m, where $\nabla u$ stands for the gradient of $u$ on \m. Using that $\g{\nablao r}{\nablao r}=-1$ and $\g{\nablao r}{\nu}>0$, we have that
\[
\g{\nablao r}{\nu}=\sqrt{1+|\nabla u|^2}\geq 1,
\]
so that
\[
\nablao r=\nabla u-\sqrt{1+|\nabla u|^2}\nu.
\]
Moreover, from Gauss and Weingarten formulae, we get
\[
\nablao_{X}\nablao r=\nabla_X\nabla u+\sqrt{1+|\nabla u|^2}AX+\g{AX}{\nabla u}\nu-X(\sqrt{1+|\nabla u|^2})\nu
\]
for every tangent vector field $X\in T\m$. Thus, \beq \label{hessiano1} \nabla^2u(X,X)=\nablao^2r(X,X)-\sqrt{1+|\nabla u|^2}\g{AX}{X} \eeq for every $X\in T\m$, where $\nablao^2r$ and $\nabla^2u$ stand for the Hessian of $r$ and $u$
in $M$ and $\Sigma$, respectively. Tracing this expression, one gets that the Laplacian of $u$ is given by \beq \label{laplaciano1} \Delta u=\Deltao r+\nablao^2r(\nu,\nu)+nH\sqrt{1+|\nabla u|^2}, \eeq where $\Deltao r$ is the
(Lorentzian) Laplacian of $r$ and $H=-(1/n)\mathrm{tr}(A)$ is the mean curvature of \m.

On the other hand, we have the following decomposition for $X$:
\[
X=X^*-\g{X}{\nablao r}\nablao r
\]
with $X^*$ orthogonal to $\nablao r$. In particular
\beq
\label{eq6}
\g{X^*}{X^*}=\g{X}{X}+\g{X}{\nablao r}^2.
\eeq
Taking into account that
\[
\nablao_{\overline{\nabla} r}\nablao r=0
\]
one easily gets that
\[
\nablao^2r(X,X)=\nablao^2r(X^*,X^*)
\]
for every $X\in T\m$.

Assume now that $K_M(\Pi)\leq c$ for all timelike planes in $M$, and that $u<\pi/\sqrt{-c}$ on \m\ when $c<0$. Then by
Lemma \ref{comparacion1} and (\ref{eq6}) we get
that
\[
\nablao^2r(X,X)=\nablao^2r(X^*,X^*)\geq -f_c(u)\g{X^*}{X^*}=-f_c(u)(1+\g{X}{\nablao r}^2).
\]
for every unit tangent vector field $X\in T\Sigma$. Therefore, by (\ref{hessiano1}) we have that
\[
\nabla^2u(X,X)\geq -f_c(u)(1+\g{X}{\nabla u}^2)-\sqrt{1+|\nabla u|^2}\g{AX}{X}
\]
for every unit $X\in T\Sigma$. Tracing this inequality, one gets the following inequality for the Laplacian of $u$
\[
\Delta u\geq-f_c(u)(n+|\nabla u|^2)+nH\sqrt{1+|\nabla u|^2}.
\]
We summarize this in the following result.
\begin{proposition}
\label{p1}
Let $M^{n+1}$ be a spacetime such that $K_M(\Pi)\leq c$ for all timelike planes in $M$. Assume that
there exists a point $p\in M$ such that $\mathcal{I}^{+}(p)\neq\emptyset$, and let $\psi:\Sigma^n\rightarrow M^{n+1}$
be a spacelike hypersurface such that $\psi(\Sigma)\subset\mathcal{I}^{+}(p)$. Let $u$ denote the Lorentzian distance function from $p$ along the hypersurface \m,
(with  $u<\pi/\sqrt{-c}$ on \m\ when $c<0$). Then
\beq
\label{deshessiano1}
\nabla^2u(X,X)\geq -f_c(u)(1+\g{X}{\nabla u}^2)-\sqrt{1+|\nabla u|^2}\g{AX}{X}
\eeq
for every unit tangent vector $X\in T\Sigma$, and
\beq
\label{deslpalaciano1}
\Delta u\geq-f_c(u)(n+|\nabla u|^2)+nH\sqrt{1+|\nabla u|^2},
\eeq
where $H$ is the future mean curvature of \m.
\end{proposition}

On the other hand, if we assume that $K_M(\Pi)\geq c$ for all timelike planes in $M$,
the same analysis using now Lemma \ref{comparacionCHINO} instead of Lemma \ref{comparacion1} yields the following
\begin{proposition}\label{comparacion5}
Let $M^{n+1}$ be a spacetime such that $K_M(\Pi)\geq c$ for all
timelike planes in $M$. Assume that there exists a point $p\in M$
such that $\mathcal{I}^{+}(p)\neq\emptyset$, and let
$\psi:\Sigma^n\rightarrow M^{n+1}$ be a spacelike hypersurface
such that $\psi(\Sigma)\subset\mathcal{I}^{+}(p)$.
Let $u$ denote the Lorentzian distance function from $p$ along the hypersurface \m,
(with $u<\pi/\sqrt{-c}$ on \m\ when $c<0$). Then
\beq
\label{deshessiano2} \nabla^2u(X,X)\leq -f_c(u)(1+\g{X}{\nabla
u}^2)-\sqrt{1+|\nabla u|^2}\g{AX}{X}
\eeq
for every unit tangent vector $X\in T\Sigma$, and
\beq \label{deslpalaciano2} \Delta
u\leq-f_c(u)(n+|\nabla u|^2)+nH\sqrt{1+|\nabla u|^2},
\eeq
where $H$ is the future mean curvature of \m.
\end{proposition}

Finally, under the assumption $\mathrm{Ric}_M(Z,Z)\geq -nc$, $c\in\R{}$, for every unit timelike vector $Z$,
Lemma \ref{comparacion3} and (\ref{laplaciano1}) lead us to the following result.
\begin{proposition}
\label{comparacion4}
Let $M^{n+1}$ be an $(n+1)$-dimensional spacetime such that
\[
\mathrm{Ric}_M(Z,Z)\geq -nc, \quad c\in\R{},
\]
for every unit timelike vector $Z$. Assume that there exists a point $p\in M$ such that
$\mathcal{I}^{+}(p)\neq\emptyset$, and let $\psi:\Sigma^n\rightarrow M^{n+1}$ be a spacelike hypersurface
such that $\psi(\Sigma)\subset\mathcal{I}^{+}(p)$. Let $u$ denote the Lorentzian distance function from $p$ along the hypersurface \m,
(with $u<\pi/\sqrt{-c}$ on \m\ when $c<0$). Then
\[
\Delta u\geq -nf_c(u)+\nablao^2d_p(\nu,\nu)+nH\sqrt{1+|\nabla u|^2},
\]
where $\nu$ and $H$ are the future-directed Gauss map and the future mean curvature of \m, respectively.
\end{proposition}

\section{Hypersurfaces bounded by a level set of the Lorentzian distance from a point}

Under suitable bounds for the sectional curvatures of the ambient spacetime,
we compare in this section the mean curvature of this hypersurface with the mean curvature of the level sets
of the Lorentzian distance in the Lorentzian space forms.
First of all, and following the terminology introduced by Pigola, Rigoli and Setti in \cite[Definition 1.10]{PRS}, the
\textit{Omori-Yau maximum principle} is said to hold on an $n$-dimensional Riemannian manifold $\Sigma^n$ if, for any
smooth function $u\in\mathcal{C}^\infty(\Sigma)$ with $u^*=\sup_\Sigma u<+\infty$ there exists a sequence of points
$\{p_k\}_{k\in\mathbb{N}}$ in $\Sigma$ with the properties
\[
\textrm{(i)} \,\,\, u(p_k)>u^*-\frac{1}{k}, \,\,\, \textrm{(ii)} \,\,\,
|\nabla u(p_k)|<\frac{1}{k},
\textrm{ and } \textrm{(iii)} \,\,\, \Delta u(p_k)<\frac{1}{k}.
\]
Equivalently, for any $u\in\mathcal{C}^\infty(\Sigma)$ with $u_*=\inf_\Sigma u>-\infty$ there
exists a sequence of points $\{p_k\}_{k\in\mathbb{N}}$ in $\Sigma$ satisfying
\[
\text{(i)} \,\,\, u(p_k)<u_*+\frac{1}{k}, \,\,\, \text{(ii)} \,\,\, |\nabla u(p_k)|<\frac{1}{k},
\text{ and } \text{(iii)} \,\,\, \Delta u(p_k)>-\frac{1}{k}.
\]
In this sense, the classical maximum principle given by Omori \cite{Om} and Yau \cite{Y} states that
the Omori-Yau maximum principle holds on every complete Riemannian manifold with Ricci curvature bounded from below.

More generally, as shown by Pigola, Rigoli and Setti \cite[Example 1.13]{PRS}, a sufficiently controlled decay of the
radial Ricci curvature of the form
\[
\mathrm{Ric}_\Sigma(\nabla\varrho,\nabla\varrho)\geq-C^2G(\varrho)
\]
where $\varrho$ is the distance function on \m\ to a fixed point, $C$ is a positive constant, and
$G:[0,+\infty)\rightarrow\mathbb{R}$ is a smooth function satisfying
\[
\textrm{(i)} \,\,\, G(0)>0, \,\,\, \textrm{(ii)} \,\,\, G'(t)\geq 0,  \,\,\, \textrm{(iii)} \,\,\,
\int_0^{+\infty}1/\sqrt{G(t)}=+\infty \textrm{ and }
\]
\[
\textrm{(iv)} \,\,\, \limsup_{t\rightarrow+\infty}tG(\sqrt{t})/G(t)<+\infty,
\]
suffices to imply the validity of the Omori-Yau maximum principle. In particular, and following the terminology
introduced by Bessa and Costa in \cite{BC}, the Omori-Yau maximum principle holds on a complete Riemannian
manifold whose Ricci curvature has \textit{strong quadratic decay} \cite{CX}, that is, with
\[
\mathrm{Ric}_\Sigma\geq-C^2(1+\varrho^2\log^2(\varrho+2)).
\]
On the other hand, as observed also by Pigola, Rigoli and Setti in \cite{PRS}, the validity
of Omori-Yau maximum principle on $\Sigma^n$ does not depend on curvature bounds  as much as one would expect. For
instance, the Omori-Yau maximum principle holds on every Riemannian manifold admitting a non-negative $C^2$ function
$\varphi$ satisfying the following requirements: (i) $\varphi(p)\rightarrow +\infty$ as $p\rightarrow \infty$;
(ii) there exists $A>0$ such that $|\nabla\varphi|\leq A\sqrt{\varphi}$ off a compact set; and (iii) there exists
$B>0$ such that $\Delta\varphi\leq B\sqrt{\varphi}\sqrt{G(\sqrt{\varphi})}$ off a compact set, where $G$ is as above
(see \cite[Theorem 1.9]{PRS}).

Now we are ready to give our first result.
\begin{theorem}
\label{mainTh1UpperB}
Let $M^{n+1}$ be an $(n+1)$-dimensional spacetime such that
\[
\mathrm{Ric}_M(Z,Z)\geq -nc, \quad c\in\R{},
\]
for every unit timelike vector $Z$. Let $p\in M$ be such that $\mathcal{I}^+(p)\neq\emptyset$, and let $\psi:\m\rightarrow M^{n+1}$
be a spacelike hypersurface such that $\psi(\m)\subset\mathcal{I}^+(p)\cap B^+(p,\delta)$ for some $\delta>0$
(with $\delta\leq\pi/\sqrt{-c}$ when $c<0$), where $B^+(p,\delta)$ denotes the future inner ball of radius $\delta$,
\[
B^+(p,\delta)=\{ q\in I^+(p) : d_p(q)<\delta \}.
\]
If the Omori-Yau maximum principle holds on \m, then its future mean curvature $H$ satisfies
\[
\inf_{\Sigma} H \leq f_c(\sup_\Sigma u),
\]
where $u$ denotes the Lorentzian distance $d_p$ along the hypersurface.
\end{theorem}
\begin{proof}
As $\mathrm{Ric}_M(Z,Z)\geq -n c$, by Proposition
\ref{comparacion4} we have that
\[
\Delta u\geq -nf_c(u)+\nablao^2r(\nu,\nu)+nH\sqrt{1+|\nabla u|^2}.
\]
Now, by applying the Omori-Yau maximum principle, there exists a sequence of points
$\{p_k\}_{k\in\mathbb{N}}$ in $\Sigma$ such that
\[
|\nabla u(p_k)|<\frac{1}{k}, \quad
\Delta u(p_k)<\frac{1}{k}, \quad
\sup_\Sigma u -\frac{1}{k}< u(p_k)\leq\sup_\Sigma u\leq\delta.
\]
Therefore
\begin{equation*}
\frac{1}{k}>\Delta u(p_k)\geq-nf_c(u(p_k))+
\nablao^2r(\nu(p_k),\nu(p_k))+
nH(p_k)\sqrt{1+|\nabla u(p_k)|^2},
\end{equation*}
and
\begin{equation}
\label{deslapla}
\inf_{\Sigma} H\leq H(p_k)
\leq\frac{1/k+n\,f_c(u(p_k))-\nablao^2r(\nu(p_k),\nu(p_k))}{n\sqrt{1+|\nabla u(p_k)|^2}}.
\end{equation}

On the other hand, we have the following decomposition for
$\nu(p_k)$:
\begin{displaymath}
\nu(p_k)=\nu^*(p_k)-\g{\nu(p_k)}{\nablao r(p_k)}\nablao r(p_k),
\end{displaymath}
with $\nu^*(p_k)$ orthogonal to $\nablao r(p_k)$.
Since
\[
\g{\nablao r(p_k)}{\nablao r(p_k)}=\g{\nu(p_k)}{\nu(p_k)}=-1
\]
and
\[
\nablao r(p_k)=\nabla u(p_k)-\g{\nablao r(p_k)}{\nu(p_k)}\nu(p_k),
\]
we have $|\nu^*(p_k)|^2=|\nabla u(p_k)|^2$
and $\lim_{\varepsilon \to 0} |\nu^*(p_k)|^2=0$. That is,
$\lim_{\varepsilon \to 0} \nu^*(p_k)=0$.\par

Now, taking into account that $\nablao^2r(\nu(p_k),\nu(p_k))=\nablao^2r(\nu^*(p_k),\nu^*(p_k))$ and making
$k\rightarrow\infty$ in (\ref{deslapla}), we conclude that
\begin{equation*}
\inf_{\Sigma} H\leq \lim_{k \to\infty}H(p_k) \leq f_c(\sup_{\Sigma} u).
\end{equation*}
\end{proof}

On the other hand, under the assumption that the sectional curvatures of timelike planes in $M$ are bounded from below we
derive the following.
\begin{theorem}
\label{mainTh1LowerB}
Let $M^{n+1}$ be an $(n+1)$-dimensional spacetime such that $K_M(\Pi) \geq c$, $c\in\R{}$, for all timelike planes
in $M$. Let $p\in M$ be such that $\mathcal{I}^+(p)\neq\emptyset$, and let $\psi:\m\rightarrow M^{n+1}$
be a spacelike hypersurface such that $\psi(\m)\subset\mathcal{I}^+(p)$.
If the Omori-Yau maximum principle holds on \m\ (and $\inf_\Sigma u<\pi/\sqrt{-c}$ when $c<0$), then its future mean
curvature $H$ satisfies
\[
\sup_{\Sigma} H \geq f_c(\inf_\Sigma u),
\]
where $u$ denotes the Lorentzian distance $d_p$ along the hypersurface. In particular, if $\inf_\Sigma u=0$ then
$\sup_{\Sigma} H=+\infty$.
\end{theorem}
\begin{proof}
We start by applying the Omori-Yau maximum principle to the positive function $u$, with $\inf_\Sigma u\geq 0$.
Therefore, there exists a sequence of points $\{p_k\}_{k\in\mathbb{N}}$ in $\Sigma$ such that
\[
|\nabla u(p_k)|<\frac{1}{k}, \quad
\Delta u(p_k)>-\frac{1}{k}, \quad
0\leq\inf_\Sigma u\leq u(p_k)<\inf_\Sigma u+\frac{1}{k}.
\]
Recall that, when $c<0$, we are assuming that $\inf_\Sigma u<\pi/\sqrt{-c}$. Thus, if $k$ is big enough
we have that $u(p_k)<\pi/\sqrt{-c}$. Therefore, the inequality (\ref{deslpalaciano2}) in
Proposition \ref{comparacion5} holds at $p_k$ and we obtain that
\[
-\frac{1}{k}<\Delta u(p_k)\leq
-f_c(u(p_k))(n+|\nabla u(p_k)|^2)+nH(p_k)\sqrt{1+|\nabla u(p_k)|^2}
\]
for $k$ big enough. It follows from here that
\begin{equation}
\label{deslaplaB}
\sup_{\Sigma} H\geq H(p_k)\geq
\frac{-1/k+f_c(u(p_k))(n+|\nabla u(p_k)|^2)}{n\sqrt{1+|\nabla u(p_k)|^2}},
\end{equation}
and making $k\rightarrow\infty$ we conclude the result. The last assertion follows from the fact that
$\lim_{s\rightarrow 0}f_c(s)=+\infty$.
\end{proof}
As a direct application of Theorem \ref{mainTh1LowerB} we get the following.
\begin{corollary}
\label{coro1}
Under the assumptions of Theorem \ref{mainTh1LowerB}, if the Omori-Yau maximum principle holds on \m\ and its
future mean curvature $H$ is bounded from above on \m,
then there exists some $\delta>0$ such that $\psi(\m)\subset O^+(p,\delta)$, where $O^+(p,\delta)$ denotes the future
outer ball of radius $\delta$,
\[
O^+(p,\delta)=\{ q\in I^+(p) : d_p(q)>\delta \}.
\]
\end{corollary}
For a proof, simply observe that $\sup_\Sigma H<+\infty$ implies that $\inf_\Sigma u>0$. This result, as well as the next
ones, has a specially illustrative consequence when the ambient is the Lorentz-Minkowski spacetime (see Remark
\ref{remark1} at the end of this section).
\begin{corollary}
\label{coro2}
Under the assumptions of Theorem \ref{mainTh1LowerB}, when $c\geq 0$ there exists no
spacelike hypersurface \m\ contained in $\mathcal{I}^+(p)$
on which the Omori-Yau maximum principle holds and having $H\leq\sqrt{c}$ on \m.
When $c<0$, there exists no
spacelike hypersurface \m\ contained in $\mathcal{I}^+(p)$
on which the Omori-Yau maximum principle holds and having $\inf_\Sigma u<\pi/(2\sqrt{-c})$ and $H\leq 0$ on \m.
\end{corollary}
In fact, when $c\geq 0$ our Theorem \ref{mainTh1LowerB} implies that for every
spacelike hypersurface \m\ contained in $\mathcal{I}^+(p)$
on which the Omori-Yau maximum principle holds, it holds that
\[
\sup_\Sigma H\geq f_c(\inf_\Sigma u)>\lim_{s\rightarrow+\infty}f_c(s)=\sqrt{c}.
\]
Therefore, it cannot happen $\sup_\Sigma H\leq\sqrt{c}$. On the other hand, when $c<0$ our Theorem \ref{mainTh1LowerB}
also implies that every spacelike hypersurface \m\ contained in $\mathcal{I}^+(p)$, with
$\inf_\Sigma u<\pi/(2\sqrt{-c})$,
on which the Omori-Yau maximum principle holds satisfies
\[
\sup_\Sigma H\geq f_c(\inf_\Sigma u)>f_c(\pi/(2\sqrt{-c}))=0.
\]
Therefore, it cannot happen $\sup_\Sigma H\leq 0$.

In particular, when the ambient spacetime is a Lorentzian space form, by putting together Theorems
\ref{mainTh1UpperB} and \ref{mainTh1LowerB}, we derive the following consequence.
\begin{theorem}
\label{mainTh2UpperLowB}
Let $M^{n+1}_c$ be a Lorentzian space form of constant sectional curvature $c$ and let $p\in M^{n+1}_c$. Let us
consider $\psi:\m\rightarrow M^{n+1}_c$ a spacelike hypersurface  such that
$\psi(\m)\subset\mathcal{I}^+(p)\cap B^+(p,\delta)$ for some $\delta>0$
(with $\delta\leq\pi/\sqrt{-c}$ if $c<0$). If the Omori-Yau maximum principle holds on \m, then its future mean
curvature $H$ satisfies
\[
\inf_{\Sigma} H \leq f_c(\sup_\Sigma u)\leq f_c(\inf_\Sigma u)\leq \sup_\Sigma H,
\]
where $u$ denotes the Lorentzian distance $d_p$ along the hypersurface.
\end{theorem}

As is well known, the curvature tensor $R$ of \m\ can be described in terms of
$R_M$, the curvature tensor of the ambient spacetime, and the shape operator of \m\ by the so called Gauss equation,
which can be written as
\beq
\label{gausseq}
R(X,Y)Z=(R_M(X,Y)Z)^\top+\g{AX}{Z}AY-\g{AY}{Z}AX
\eeq
for all tangent vector fields $X,Y,Z\in T\m$, where $(R_M(X,Y)Z)^\top$ denotes the
tangential component of $R_M(X,Y)Z$. Observe that our choice here for the
curvature tensor is the one in \cite{BEE} (and the opposite to that in \cite{O'N}). Therefore, the Ricci curvature of \m\
is given by
\begin{eqnarray}
\label{ricci}
\nonumber \mathrm{Ric}(X,X) & = & \mathrm{Ric}_M(X,X)-K_M(X\wedge \nu)|X|^2+nH\g{AX}{X}+|AX|^2\\
{} & = & \mathrm{Ric}_M(X,X)-\left(K_M(X\wedge \nu)+\frac{n^2H^2}{4}\right)|X|^2+|AX+\frac{n}{2}X|^2\\
\nonumber {} & \geq & \mathrm{Ric}_M(X,X)-\left(K_M(X\wedge \nu)+\frac{n^2H^2}{4}\right)|X|^2,
\end{eqnarray}
for $X\in T\m$, where $\mathrm{Ric}_M$ stands for the Ricci curvature of the ambient spacetime and
$K_M(X\wedge \nu)$ denotes the sectional curvature of the timelike plane spanned by $X$ and $\nu$. In particular, when
$M^{n+1}_c$ is a Lorentzian space form of constant sectional curvature $c$, then $\mathrm{Ric}_M(X,X)=nc|X|^2$ for
all spacelike vector $X\in T\Sigma$, and (\ref{ricci}) reduces to
\[
\mathrm{Ric}(X,X)\geq \left((n-1)c-\frac{n^2H^2}{4}\right)|X|^2.
\]
Therefore, if $\inf_\Sigma H<-\infty$ and $\sup_\Sigma H<+\infty$ (that is, $\sup_\Sigma H^2<+\infty$), then
the Ricci curvature of $\Sigma$ is bounded from below. In particular, every spacelike hypersurface with constant mean
curvature in $M^{n+1}_c$ has Ricci curvature bounded from below. As a consequence.
\begin{corollary}
\label{corlevelsets}
Let $M^{n+1}_c$ be a Lorentzian space form of constant sectional curvature $c$ and let $p\in M^{n+1}_c$. If
\m\ is a complete spacelike hypersurface in $M^{n+1}_c$ with constant mean curvature $H$ which is contained in
$\mathcal{I}^+(p)$ and bounded from above by a level set of the Lorentzian distance function $d_p$
(with $d_p<\pi/\sqrt{-c}$ if $c<0$), then \m\ is necessarily a level set of $d_p$.
\end{corollary}
\begin{proof}
Our hypotheses imply that \m\ is contained in $\mathcal{I}^+(p)\cap B^+(p,\delta)$ for some $\delta>0$
(with $\delta\leq\pi/\sqrt{-c}$ if $c<0$), and that \m\ has Ricci curvature bounded from below by the constant
$(n-1)c-n^2H^2/4$. In particular, the Omori-Yau maximum principle holds on \m. Therefore, by Theorem \ref{mainTh2UpperLowB}
we get that
\[
H\leq f_c(\sup_\Sigma u)\leq f_c(\inf_\Sigma u)\leq H,
\]
which implies that $\sup_\Sigma u=\inf_\Sigma u=f_c^{-1}(H)$ and then \m\ is necessarily the level set
$d_p=f_c^{-1}(H)$.
\end{proof}

\begin{remark}
\label{remark1}
As observed after the proof of Corollary \ref{coro1}, our last results have specially simple and illustrative
consequences when the ambient is the Lorentz-Minkowski spacetime. Consider $\mathbb{L}^{n+1}$ the standard model of the
Lorentz-Minkowski space, that is, the real vector space $\mathbb{R}^{n+1}$ with canonical
coordinates $(x_1,\ldots,x_{n+1})$, endowed with the Lorentzian metric
\[
\g{}{}=dx_1^2+\cdots +dx_{n}^2-dx_{n+1}^2
\]
and with the time orientation determined by $e_{n+1}=(0,\ldots,0,1)$. For a given $p\in\mathbb{L}^{n+1}$, it can be
easily seen that
\[
\mathcal{I}^+(p)=\{ q\in\mathbb{L}^{n+1} : \g{q-p}{q-p}<0, \quad \mathrm{and} \quad \g{q-p}{e_{n+1}}<0 \}.
\]
The Lorentzian distance is given by $d_p(q)=\sqrt{-\g{q-p}{q-p}}$ for every $q\in\mathcal{I}^+(p)$, and the
level sets of $d_p$ are precisely the future components of the hyperbolic spaces centered at $p$. Also, observe that the
boundary of $\mathcal{I}^+(p)$ is nothing but the future component of the lightcone with vertex at $p$.

Then, Corollary \ref{coro1} implies that every complete spacelike hypersurface contained in $\mathcal{I}^+(p)$ and having
bounded mean curvature is bounded away from the lightcone, in the sense that there exists some $\delta>0$ such that
\[
\g{q-p}{q-p}\leq -\delta^2<0
\]
for every $q\in\m$.
Also, Corollary \ref{coro2} implies that there exists no complete spacelike hypersurface contained in $\mathcal{I}^+(p)$
and having non-positive bounded future mean curvature. In particular, there exists no complete
hypersurface with constant mean curvature $H\leq 0$ contained in $\mathcal{I}^+(p)$. Finally, Corollary \ref{corlevelsets} allows to improve Theorem 2 in
\cite{AA1} as follows.
\begin{corollary}
\label{coro3}
The only complete spacelike hypersurfaces with constant mean curvature in the Lorentz-Minkowski space $\mathbb{L}^{n+1}$
which are contained in $\mathcal{I}^+(p)$ (for some fixed $p\in\mathbb{L}^{n+1}$) and bounded from above by a hyperbolic
space centered at $p$ are precisely the hyperbolic spaces centered at $p$.
\end{corollary}
\end{remark}

\section{Analysis of the Lorentzian distance function from an achronal spacelike hypersurface}
\label{analy}
Given $N^n\subset M^{n+1}$ an achronal spacelike hypersurface, we
can define the Lorentzian distance function from $N$, $d_N:M\rightarrow [0,+\infty]$, by
\begin{displaymath}
d_N(q):=\sup \{d(p,q): p\in N\},
\end{displaymath}
for all $q\in M$. As in the previous case, to guarantee the
smoothness of $d_N$, we need to restrict this function on certain
special subsets of $M$. Let $\eta$ be the future-directed Gauss map of $N$. Then, we can define the function
$s_N:N\rightarrow [0,+\infty]$ by
\[
s_N(p)=\sup\{ t\geq 0 : d_N(\gamma_p(t))=t\},
\]
where $\gamma_p:[0,a)\rightarrow M$ is the future inextendible
geodesic starting at $p$ with initial velocity $\eta_p$. Then, we
can define
\[
\tilde{\mathcal{I}}^{+}(N)=\{ t\eta_p: \mbox{ for all } p\in N
\mbox{ and } 0<t<s_N(p) \}
\]
and consider the subset $\mathcal{I}^{+}(N)\subset M$ given by
\[
\mathcal{I}^{+}(N)=\mathrm{exp}_N(\mathrm{int}(\tilde{\mathcal{I}}^{+}(N)))\subset
I^{+}(N),
\]
where $\mathrm{exp}_N$ denotes the exponential map with respect to
the hypersurface $N$.

Observe that
\[
\mathrm{exp}_N:
\mathrm{int}(\tilde{\mathcal{I}}^{+}(N))\rightarrow
\mathcal{I}^{+}(N)
\]
is a diffeomorphism and $\mathcal{I}^{+}(N)$ is an open subset
(possible empty). In the next auxiliary result we collect some interesting properties about $d_N$ (see \cite[Section 3.2]{EGK}).
\begin{lemma}
\label{lemaEduardo2} Let $N$ be an achronal spacelike hyersurface
in a spacetime $M$.
\begin{enumerate}
\item If $N$ is compact and $(M,g)$ is globally hyperbolic, then $s_N(p)>0$ for all $p\in N$ and
$\mathcal{I}^{+}(N)\neq\emptyset$. \item If
$\mathcal{I}^{+}(N)\neq\emptyset$, then $d_N$ is smooth on
$\mathcal{I}^{+}(N)$ and its gradient $\nablao d_N$ is a
past-directed timelike (geodesic) unit vector field on
$\mathcal{I}^{+}(N)$.
\end{enumerate}
\end{lemma}

To state our results concerning the Lorentzian distance function from an achronal spacelike hypersurface, we need to introduce the following concepts.
\begin{definition}
Let $N^n$ be a spacelike hypersurface in $M^{n+1}$ with future-directed Gauss map $\eta$. For all $p\in N$, let $\gamma_p$ be the
normal future-directed unit timelike geodesic with $\gamma_p(0)=p$ and $\gamma_p'(0)=\eta_p$. A normal Jacobi vector field along
$\gamma_p$ is said to be $N$-Jacobian if $J'(0)=-A_N(J(0))$, where
$A_N$ denotes the shape operator of $N$ with respect to $\eta$.
\end{definition}
\begin{definition}
Let $N^n$ be a spacelike hypersurface of $M^{n+1}$ with future-directed Gauss map $\eta$ and let $\gamma:[0,s]\rightarrow M$
be a future-directed unit timelike geodesic orthogonal to $N$. If $X$ and $Y$
are vector fields along $\gamma$, the index form of the geodesic
$\gamma$ with respect to $N$ is given by
\begin{displaymath}
I_N(X,Y)=-\int_0^s (\g{X'}{Y'}-\g{R(X,\gamma')\gamma'}{Y})dt
+\g{A_N X}{Y}
\end{displaymath}
where $A_N$ denotes the shape operator of $N$ with respect to $\eta$. $I_N$ defines a bilinear form on the space of vector fields $X$, $Y$ orthogonal to $\gamma$.
\end{definition}

\begin{remark} Consider the standard model of a simply connected complete Lorentzian space form
\[
M^{n+1}_c=
\begin{cases}
\mathbb{S}^{n+1}_1(1/\sqrt{c})=\{ x\in \mathbb{R}^{n+2}_1: \g{x}{x}=1/c\}&\text{if $c>0$}\\
\mathbb{L}^{n+1}=\mathbb{R}^{n+1}_1 &\text{if $c=0$}\\
\mathbb{H}^{n+1}_1(1/\sqrt{-c})=\{ x\in \mathbb{R}^{n+2}_2:
\g{x}{x}=-1/c\}&\text{if $c<0$}.
\end{cases}
\]
In these space forms, we have the following preferred spacelike hypersurfaces, which are totally geodesic in $M^{n+1}_c$:
\[
N_c=
\begin{cases}
\mathbb{S}^{n}(1/\sqrt{c})=\{ x\in \mathbb{S}^{n+1}_1(1/\sqrt{c}): x_{n+2}=0\}&\text{if $c>0$}\\
\mathbb{R}^{n}=\{ x\in \mathbb{L}^{n+1}: x_{n+1}=0\} &\text{if $c=0$}\\
\mathbb{H}^{n}(1/\sqrt{-c})=\{ x\in
\mathbb{H}^{n+1}_1(1/\sqrt{-c}): x_{n+1}>0, x_{n+2}=0\}&\text{if
$c<0$}.
\end{cases}
\]
Observe that when $c\geq 0$, $M_c^{n+1}$ is globally hyperbolic and the hypersurfaces $N_c$ are Cauchy hypersurfaces with $\mathcal{I}^{+}(N_c)=I^+(N_c)$.

When we consider the totally geodesic hypersurfaces $N_c$ immersed
in the Lorentzian space forms $M_c^{n+1}$, it is easy to compute,
using the Jacobi equations, the $N_c$-Jacobi fields. To see it, consider $\gamma_c:[0,s]\rightarrow M_c$ a future-directed, unit timelike geodesic
emanating orthogonally from $N_c$ (with $s<\pi/(2\sqrt{-c})$ when $c<0$). Then the $N_c$-Jacobi field $J_c(t)$
is given by $J_c(t)=\mathrm{c}_c(t) E_c(t)$, where $E_c(t)$ is a normal parallel vector field along $\gamma_c$,  and
$$
\mathrm{c}_c(t) = \begin{cases} \frac{\cosh(\sqrt{c}\,t)}{\cosh(\sqrt{c}\,s)}&\text{if $c>0$ and $0\leq t\leq s$}\\
\phantom{\sqrt{c}} 1&\text{if $c=0$ and $0\leq t\leq s$}\\
\frac{\cos(\sqrt{-c}\,t)}{\cos(\sqrt{-c}\,s)}&\text{if $c<0$ and $0\leq t\leq s<\pi/(2\sqrt{-c})$} .\end{cases}
$$
In this case, the index form $I_{N_c}$ acting on the $N_c$-Jacobi fields is given by
\begin{eqnarray}
\label{eq9}
\nonumber I_{N_c}(J_c,J_c) & = & -\int_0^s (\g{J'_c}{J'_c}+c \g{J_c}{J_c})dt\\
{} & = & -\int_0^s (\mathrm{c}'_c(t)^2+c \mathrm{c}_c(t)^2)\g{E_c(t)}{E_c(t)}dt=-F_c(s)\g{J_c(s)}{J_c(s)},
\end{eqnarray}
since $\g{E_c(t)}{E_c(t)}=\g{E_c(s)}{E_c(s)}=\g{J_c(s)}{J_c(s)}$ is constant, where
\[
F_c(s)=
\begin{cases}
\sqrt{c}\tanh(\sqrt{c}\,s) &\text{if $c>0$ and $s>0$}\\
\phantom{\sqrt{c}} 0 &\text{if $c=0$ and $s>0$}\\
-\sqrt{-c}\tan(\sqrt{-c}\,s) &\text{if $c<0$ and $0<s<\pi/(2\sqrt{-c})$}.
\end{cases}
\]
\end{remark}

\begin{theorem}\label{maximality}(Lorentzian version of \cite[Theorem 32.1.1]{Bu}) Let $N^n$ be a spacelike hypersurface of $M^{n+1}$ with future-directed
Gauss map $\eta$. Given $p\in N$, let $\gamma:[0,s]\rightarrow M$
be the normal future-directed unit timelike geodesic with
$\gamma(0)=p$ and $\gamma'(0)=\eta_p$, and suposse there are no
points on $\gamma$ which are focal along $\gamma$ to $N$. Let $J$
be an $N$-Jacobian vector filed along $\gamma$. Then, for every
vector field $X$ along $\gamma$ which is not identically zero and
orthogonal to $\gamma$ such that $X(s)=J(s)$ it holds
\begin{displaymath}
I_N(J,J)\geq I_N(X,X),
\end{displaymath}
with equality if and only if $J=X$.
\end{theorem}
\begin{proof}
The proof follows the ideas of the proof of Theorem 32.1.1 of
\cite{Bu}, taking into account that $\gamma$ being timelike,
equation $(2)$ in \cite[pag. 237]{Bu} becomes
\begin{displaymath}
I_N(X,X)=I_N(J,J)-\int_0^s \g{A}{A} dt,
\end{displaymath}
where in our case $A$ is a spacelike vector field along $\gamma$.
In particular, $\g{A}{A}\geq 0$ and then $I_N(J,J)\geq I_N(X,X)$.
\end{proof}

Using this, we can state the following comparison results for the Hessian of the
function $d_N$.
\begin{lemma}
\label{comparacion1achro} Let $M^{n+1}$ be an $(n+1)$-dimensional
spacetime such that $K_M(\Pi)\leq c$, $c\in \mathbb{R}$, for all timelike
planes in $M$. Let $N^n\subset M$ be an achronal spacelike
hypersurface with positive semi-definite second fundamental form (with respect to the
unique future-directed timelike unit normal field) and such that $\mathcal{I}^{+}(N)\neq \emptyset$.
Let $q\in \mathcal{I}^{+}(N)$ (with $d_N(q)<\pi/(2\sqrt{-c})$ when $c<0$).  Then, for every spacelike vector
$x\in T_{q}M$ orthogonal to $\nablao d_N(q)$ it holds that
\[
\nablao^2d_N(x,x)\geq -F_c(d_N(q))\g{x}{x},
\]
where $\nablao^2$ stands for the Hessian operator on $M$. When $c<0$ but $d_N (q)\geq \pi/(2\sqrt{-c})$, then it still holds that
\[
\nablao^2d_N(x,x)\geq 0.
\]
\end{lemma}
\begin{proof}
The proof follows the ideas of the proof of Lemma \ref{comparacion1}. Given $q\in \mathcal{I}^+(N)$, there exists $p \in N$ such that
$q=\mathrm{exp}_N(s \eta_p)$ with $s=d_N(q)$. Let $\gamma:[0,s]\rightarrow M$ be the normal future-directed unit timelike geodesic with
$\gamma(0)=p$ and $\gamma'(0)=\eta_p$.  From \cite[Proposition 3.7]{EGK}, we know that

\[
\nablao^2d_N(x,x)=I_N(J,J),
\]
where $J$ is the $N$-Jacobi field along $\gamma$ with $J(s)=x$. By Theorem \ref{maximality}  we get that
\begin{equation}
\label{eq7}
\nablao^2d_N(x,x)=I_N(J,J)\geq I_N(X,X),
\end{equation}
for every normal vector field $X$ along $\gamma$ such that  $X(s)=J(s)=x$. Assume now that $s<\pi/(2\sqrt{-c})$ when $c<0$, and define $X(t)=\mathrm{c}_c(t)Y(t)$, where $Y(t)$ is the (unique) parallel vector field along $\gamma$
with $Y(s)=x$. From (\ref{eq7}) we obtain that
\begin{eqnarray*}
\nablao^2 d_N(X,X) & \geq & -\int_0^s(\g{X'(t)}{X'(t)}-\g{R(X(t),\gamma'(t))\gamma'(t)}{X(t)})dt\\
{} & {} & +\g{A_N (X(0))}{X(0)}\\
{} &\geq&  -\int_0^s(\g{X'(t)}{X'(t)}+c\g{X(t)}{X(t)})dt = -F_c(s)\g{x}{x}.
\end{eqnarray*}

Finally, when $c<0$ but $d_p(q)\geq\pi/(2\sqrt{-c})$, then $K_M(\Pi)\leq c<0$ and we get that
\[
\nablao^2d_N(x,x)\geq-F_0(d_p(q))\g{x}{x}=0.
\]
\end{proof}
\begin{lemma}
\label{comparacionCHINOachro} Let $M^{n+1}$ be an
$(n+1)$-dimensional spacetime such that $K_M(\Pi)\geq c$, $c\in \mathbb{R}$, for all timelike planes in $M$. Let $N^n\subset M$ be an
achronal spacelike hypersurface with negative semi-definite second fundamental form (with
respect to the unique future-directed timelike unit normal field)
 and such that $\mathcal{I}^{+}(N)\neq
\emptyset$. Let $q\in \mathcal{I}^{+}(N)$ (with $d_N(q)<\pi/(2\sqrt{-c})$ when $c<0$).  Then, for every
spacelike vector $x\in T_{q}M$ orthogonal to $\nablao d_N(q)$ it
holds that
\[
\nablao^2d_N(x,x)\leq -F_c(d_N(q))\g{x}{x},
\]
where $\nablao^2$ stands for the Hessian operator on $M$.
\end{lemma}
\begin{proof}
Similarly, the proof follows the ideas of the proof of Lemma \ref{comparacionCHINO}.
As in the previous proof, given $q\in \mathcal{I}^+(N)$, there exists $p \in N$ such that
$q=\mathrm{exp}_N(s \eta_p)$ with $s=d_N(q)$. Let $\gamma:[0,s]\rightarrow M$ be the normal future-directed unit timelike geodesic with
$\gamma(0)=p$ and $\gamma'(0)=\eta_p$. From \cite[Proposition 3.7]{EGK}, we know that
\begin{eqnarray}
\label{eq8}
\nonumber \nablao^2d_N(x,x) & = & -\int_0^s(\g{J'(t)}{J'(t)}-\g{R(J(t),\gamma'(t))\gamma'(t)}{J(t)})dt\\
\nonumber {} & {} &  +\g{A_N(J(0))}{J(0)}\\
{} &\leq& -\int_0^s(\g{J'(t)}{J'(t)}+c\g{J(t)}{J(t)})dt,
\end{eqnarray}
where $J$ is the $N$-Jacobi field along $\gamma$ such that $J(s)=x$.

Let $\{E_1(t),\ldots,E_{n+1}(t)\}$ be an orthonormal frame of parallel vector fields along $\gamma$ such that
$E_{n+1}=\gamma'$. Write $J(t)=\sum_{i=1}^n\lambda_i(t)E_i(t)$. Consider $\gamma_c:[0,s]\rightarrow M^{n+1}_c$ a future directed timelike unit geodesic in the Lorentzian
space form of constant curvature $c$ orthogonal to $N_c$, and let $\{E^c_1(t),\ldots,E^c_{n+1}(t)\}$ be an orthonormal frame of parallel
vector fields along $\gamma_c$ such that $E^c_{n+1}=\gamma'_c$. Define $X_c(t)=\sum_{i=1}^n\lambda_i(t)E^c_i(t)$, and
observe that
\begin{eqnarray*}
\g{J'(t)}{J'(t)}+c\g{J(t)}{J(t)} & = & \g{X'_c}{X'_c}_c-\g{R_c(X_c,\gamma'_c)\gamma'_c}{X_c}_c,
\end{eqnarray*}
where $\g{}{}_c$ and $R_c$ stand for the metric and Riemannian tensors of $M^{n+1}_c$. Then, (\ref{eq8}) becomes
\begin{displaymath}
\nablao^2d_N(x,x)\leq I_{N_c}(X_c,X_c).
\end{displaymath}
Since there are no focal points of $\gamma_c(0)$ along $\gamma_c$ (recall that $s<\pi/(2\sqrt{-c})$ when $c<0$), by Theorem \ref{maximality}
\begin{equation}
I_{N_c}(X_c,X_c)\leq I_{N_c}(J_c,J_c),
\end{equation}
where $J_c$ stands for the $N_c$-Jacobi field along $\gamma_c$ such that  $J_c(s)=X_c(s)$, and by (\ref{eq9}) we conclude
\[
\nablao^2 d_N(x,x)\leq I_{N_c}(J_c,J_c)=-F_c(s)\g{X_c(s)}{X_c(s)}=-F_c(d_N(q))\g{x}{x}.
\]
\end{proof}

\begin{lemma}
\label{comparacion3achro}
Let $M^{n+1}$ be an $(n+1)$-dimensional spacetime such that
\[
\mathrm{Ric}_M(Z,Z)\geq -nc, \quad c\in\R{},
\]
for every unit timelike vector $Z$. Let $N^n\subset M$ be an achronal
spacelike hypersurface such that $\mathcal{I}^{+}(N)\neq
\emptyset$ and let $q\in \mathcal{I}^{+}(N)$
(with  $d_N(q)<\pi/(2\sqrt{-c})$ when $c<0$). Then
\begin{displaymath}
\Deltao d_N(q)\geq -nF_c(d_N(q))- n \mathrm{c}_c(0)^2 H_N (p),
\end{displaymath}
where $\Deltao$ stands for the (Lorentzian) Laplacian operator on $M$, $H_N$ is the mean curvature of the hypersurface $N$ with respect to the future-directed Gauss map $\eta$, and
$p$ is the orthogonal projection of $q$ on $N$. When $c<0$ but $d_N(q)\geq\pi/(2\sqrt{-c})$, then it
still holds that
\begin{displaymath}
\Deltao d_p(q)\geq -n H_N(p).
\end{displaymath}
\end{lemma}
Here, by $p$ being the orthogonal projection of $q$ on $N$, we mean that $p$ is the (unique) point of $N$ such that $q=\mathrm{exp}_N (d_N(q)\eta_p)$.
\begin{proof}
The proof follows the ideas of the proof of Lemma \ref{comparacion3}. Let $\gamma:[0,s]\rightarrow M$ be the normal future-directed unit timelike geodesic with
$\gamma(0)=p$ and $\gamma'(0)=\eta_p$. Let $\{e_1,\ldots,e_{n}\}$ be orthonormal vectors in $T_{q}M$ orthogonal to
$\gamma'(s)=-\nablao d_N(q)$, so that
\beq
\label{eq10}
\Deltao d_N(q)=\sum_{j=1}^n\nablao^2 d_N(e_j,e_j).
\eeq
As in the proof of Lemma \ref{comparacion1achro}, we have that, for every $j=0,\ldots,n$,
\[
\nablao^2d_N(e_j,e_j)\geq I_N(X_j,X_j)
\]
for every normal vector field $X_j$ along $\gamma$ such that $X_j(s)=e_j$ ,
which implies that
\beq
\label{eq11}
\Deltao d_N(q)\geq\sum_{j=1}^n I_N(X_j,X_j).
\eeq

Assume now that $s=d_N(q)<\pi/(2\sqrt{-c})$ when $c<0$, and let $\{E_1(t),\ldots,E_{n+1}(t)\}$ be an orthonormal frame of parallel vector fields along $\gamma$ such that
$E_j(s)=e_j$ for every $j=0,\ldots,n$, and $E_{n+1}=\gamma'$.

Define
\[
X_j(t)=\mathrm{c}_c(t)E_j(t), \quad j=1,\ldots,n.
\]
Then

\begin{eqnarray*}
\sum_{j=1}^n I_\gamma(X_j,X_j) & = & -n\int_0^s
\left(\mathrm{c}'_c(t)^2-\frac{\mathrm{c}_c(t)^2}{n}\mathrm{Ric}_M(\gamma'(t),\gamma'(t))\right)dt \\
{} & {} & + \sum_{j=1}^n\g{A_N(X_j(0))}{X_j(0)} \\
{} & \geq & -n\int_0^s\left(\mathrm{c}'_c(t)^2+c\mathrm{c}_c(t)^2\right)dt +\mathrm{c}_c(0)^2\sum_{j=1}^n\g{A_N(E_j(0))}{E_j(0)}\\
{}&=&-nF_c(s)-n \mathrm{c}_c(0)^2 H_N(p).
\end{eqnarray*}

Finally, when $c<0$ but $d_p(q)\geq\pi/(2\sqrt{-c})$, then $\mathrm{Ric}_M(Z,Z)\geq -nc>0$ and
\[
\Deltao d_N(q)\geq -nF_0(d_p(q))-n \mathrm{c}_0(0)^2 H_N(p)=-n H_N(p).
\]
\end{proof}

Now, let $\psi:\Sigma^n\rightarrow M^{n+1}$ be a spacelike
hypersurface immersed into the spacetime $M$ such that
$\psi(\Sigma)\subset\mathcal{I}^{+}(N)\neq\emptyset$ and let
$v=d_N\circ\psi:\m\rightarrow(0,\infty)$ be the function $d_N$
along the hypersurface. Using the same arguments as in Section \ref{comp} for the Lorentzian distance from a point,
we can state the following bounds for the Hessian and the Laplacian of the function $d_N$ along the hypersurface $\Sigma$,
under appropriate assumptions on the curvature of the spacetime $M$.
\begin{proposition}
\label{p1achro}
Let $M^{n+1}$ be a spacetime such that $K_M(\Pi)\leq c$ ,(resp. $K_M(\Pi)\geq c$) for all timelike planes in $M$.
Assume that there exists an achronal spacelike hypersurface $N^n\subset M$ with positive (resp. negative) semi-definite future second fundamental form and
$\mathcal{I}^{+}(N)\neq\emptyset$, and let $\psi:\Sigma^n\rightarrow M^{n+1}$ be a spacelike hypersurface such that $\psi(\Sigma)\subset\mathcal{I}^{+}(N)$.
Let $v$ denote the Lorentzian distance function from $N$ along the hypersurface \m, (with  $v<\pi/(2\sqrt{-c})$ on \m\ when $c<0$). Then
\beq
\label{deshessiano1achro}
\nabla^2v(X,X)\geq  (\mathrm{resp.} \leq) -F_c(v)(1+\g{X}{\nabla v}^2)-\sqrt{1+|\nabla v|^2}\g{AX}{X}
\eeq
for every unit tangent vector $X\in T\Sigma$, and
\beq
\label{deslpalaciano1achro}
\Delta v\geq   (\mathrm{resp.} \leq) -F_c(v)(n+|\nabla v|^2)+nH\sqrt{1+|\nabla v|^2},
\eeq
where $A$ and $H$ are the future shape operator and the future mean curvature of \m, respectively.
\end{proposition}
\begin{proposition}
\label{p2achro}
Let $M^{n+1}$ be an $(n+1)$-dimensional spacetime such that
\[
\mathrm{Ric}_M(Z,Z)\geq -nc, \quad c\in\R{},
\]
for every unit timelike vector $Z$. Assume that there exists an achronal spacelike
hypersurface $N^n\subset M$ with
$\mathcal{I}^{+}(N)\neq\emptyset$, and let $\psi:\Sigma^n\rightarrow M^{n+1}$ be a spacelike hypersurface such that $\psi(\Sigma)\subset\mathcal{I}^{+}(N)$.
Let $v$ denote the Lorentzian distance function from $N$ along the hypersurface \m, (with $v<\pi/(2\sqrt{-c})$ on \m\ when $c<0$). Then
\[
\Delta v\geq -nF_c(v)+\nablao^2d_N(\nu,\nu)+nH\sqrt{1+|\nabla v|^2}-n\mathrm{c}_c(0)^2H_N,
\]
where $\nu$ and $H$ are the future-directed Gauss map and the future mean curvature of \m, respectively, and $H_N$ stands for the future mean curvature of $N$ along the orthogonal projection of \m\ on $N$.
\end{proposition}
The proof of Propositions \ref{p1achro} and \ref{p2achro} parallels that of Propositions \ref{p1}, \ref{comparacion5} and \ref{comparacion4}, simply by taking now $r=d_N$ and using the Lemmata \ref{comparacion1achro},
\ref{comparacionCHINOachro} and \ref{comparacion3achro}.

Using these inequalities,
we can obtain the following bounds for the future mean curvature of spacelike hypersurfaces. The proofs are similar to
that of Theorems \ref{mainTh1UpperB} and \ref{mainTh1LowerB}.
\begin{theorem}
\label{mainTh1UpperBachro}
Let $M^{n+1}$ be an $(n+1)$-dimensional spacetime such that
\[
\mathrm{Ric}_M(Z,Z)\geq -nc, \quad c\in\mathbb{R},
\]
for every unit timelike vector $Z$. Let $N$ be an achronal spacelike hypersurface in $M$ with $\mathcal{I}^{+}(N)\neq\emptyset$ whose future mean curvature satisfies $\sup H_N<+\infty$.
Let $\psi:\Sigma^n\rightarrow M^{n+1}$ be a
spacelike hypersurface such that  $\psi(\m)\subset\mathcal{I}^+(N)\cap B^+(N,\delta)$ for some $\delta>0$ (with $\delta\leq \pi/(2\sqrt{-c})$ when $c<0$), where
\[
B^+(N,\delta)=\{ q\in I^+(N) : d_N(q)<\delta \}.
\]
If the Omori-Yau maximum principle holds on \m, then
\[
\inf_{\Sigma} H \leq F_c(\sup_\Sigma v)+\mathrm{c}_c(0)^2\sup H_N,
\]
where $H$ is the future mean curvature of \m.
\end{theorem}
\begin{theorem}
\label{mainTh1LowerBachro}
Let $M^{n+1}$ be an $(n+1)$-dimensional spacetime such that $K_M(\Pi)\geq c$, $c\in\mathbb{R}$, for all timelike
planes in $M$. Let $N$ be an achronal spacelike hypersurface in $M$ with $\mathcal{I}^{+}(N)\neq\emptyset$ whose future second fundamental form is negative
semi-definite. Let $\psi:\Sigma^n\rightarrow M^{n+1}$ be a
spacelike hypersurface such that $\psi(\m)\subset\mathcal{I}^+(N)$.
If the Omori-Yau maximum principle holds on \m\ (and $\inf_\Sigma v<\pi/(2\sqrt{-c})$ when $c<0$), then
\[
\sup_{\Sigma} H \geq F_c(\inf_\Sigma v),
\]
where $H$ is the future mean curvature of \m.
\end{theorem}

\section{Hyperbolicity of spacelike hypersurfaces}
\label{secHyp}
In this Section we consider some function
theoretic properties satisfied by spacelike hypersurfaces with
controlled mean curvature in spacetimes with timelike sectional
curvatures bounded from below.

First of all, we are going to recall a standard characterization of hyperbolicity of a Riemannian manifold.
\begin{lemma}[\cite{Gri}]
A Riemannian manifold $\Sigma^n$ is hyperbolic if and only if it holds one of the two following equivalent conditions:
\begin{itemize}
\item[(a)] There exists a non-constant bounded (from above and from below) subharmonic function globally defined
on $\Sigma$.
\item[(b)] There exists a non-constant positive superharmonic function globally defined on $\Sigma$.
\end{itemize}
\end{lemma}
For the equivalence between a) and b), observe that if $f$ is a non-constant bounded (from above and from below)
subharmonic function on $\Sigma$, then choosing $C>\max_\Sigma f$ we obtain $C-f$ a non-constant positive superharmonic
function. Conversely, if $f$ is a non-constant positive superharmonic function on $\Sigma$, then $f/\sqrt{1+f}$ determines
a non-constant bounded (from above and from below) subharmonic function.

As a consequence of our previous results we have the following.
\begin{theorem}\label{MainTh4Hyperbolicity}
Let $M^{n+1}$ be an $(n+1)$-dimensional spacetime, $n\geq 2$, such that $K_M(\Pi)\geq c$ for all timelike planes in $M$.
Assume that there exists a point $p\in M^{n+1}$ such that $\mathcal I^+(p)\neq \emptyset$, and let
$\psi:\m\rightarrow M^{n+1}$ be a spacelike hypersurface with $\psi(\Sigma)\subset \mathcal I^+(p)$. Let us denote by
$u$ the function $d_p$ along the hypersurface, and assume that $u\leq\pi/(2\sqrt{-c})$ if $c<0$. Then
\begin{itemize}
\item[(i)] If the future mean curvature of $\Sigma$ satisfies
\begin{equation}
\label{hy5.2}
H\leq \frac{2 \sqrt{n-1}}{n} f_c(u) \quad \text{ (with $H<f_c(u)$ at some point of \m\ if $n=2$)}
\end{equation}
then $\Sigma$  is hyperbolic.
\item[(ii)] If $c=0$ and $H\leq 0$, then $\Sigma$ is hyperbolic.
\item[(iii)] If $c>0$ and $H\leq \frac{2\sqrt{n-1}}{n}\sqrt{c}$, then $\Sigma$ is hyperbolic.
\end{itemize}
In particular, every maximal hypersurface contained in $\mathcal I^+(p)$ (and satisfying $u<(\pi/2\sqrt{-c})$ if $c<0$)
is hyperbolic.
\end{theorem}
\begin{proof}
In order to prove (i), first of all, observe that $u$ is a non-constant positive function defined on \m. Otherwise,
$\Sigma$ would be an open piece of the level set given by $d_p=u$, with $\Delta u=0$ and $\nabla u=0$, and by
Proposition \ref{comparacion5} its mean
curvature would be $H\geq f_c(u)$, which
cannot happen because of (\ref{hy5.2}). Now we apply Proposition \ref{comparacion5} to get
\[
\Delta u \leq -f_c(u)(n+|\nabla u|^2)+n H\sqrt{1+|\nabla u|^2}.
\]
Observe that $x=\sqrt{n-2}$ is a minimum of the function
\[
\phi(x)=\frac{n+x^2}{n\sqrt{1+x^2}},\quad \mathrm{with} \quad x\geq 0,
\]
with $\phi(\sqrt{n-2})=2\sqrt{n-1}/n$. Therefore
\[
\frac{2\sqrt{n-1}}{n}\leq \frac{n+|\nabla u|^2}{n\sqrt{1+|\nabla u|^2}}.
\]
Since $f_c(u)\geq 0$ (recall that we assume $u\leq\pi/(2\sqrt{-c})$ if $c<0$), then our hypothesis on $H$ implies that
\[
H\leq \frac{2 \sqrt{n-1}}{n} f_c(u)\leq\frac{f_c(u)(n+|\nabla u|^2)}{n\sqrt{1+|\nabla u|^2}}.
\]
That is,
\[
nH\sqrt{1+|\nabla u|^2}\leq f_c(u)(n+|\nabla u|^2)
\]
which yields $\Delta u\leq 0$.
As a consequence, $u$ is a non-constant positive superharmonic function on $\Sigma$ and hence it is hyperbolic.

To prove (ii) and (iii), simply observe that $f_0(u)=1/u>0$ and $f_c(u)=\sqrt{c}\coth(\sqrt{c}u)>\sqrt{c}$ on
$\Sigma$.
\end{proof}

Finally, using Proposition \ref{p1achro} and following the proof of Theorem \ref{MainTh4Hyperbolicity}, we are able to conclude the following result.
\begin{theorem}\label{MainTh4Hyperbolicityachro}
Let $M^{n+1}$ be an $(n+1)$-dimensional spacetime, $n\geq 2$, such that $K_M(\Pi)\geq c\geq 0$ for all timelike planes in $M$. Assume that there exists an
achronal spacelike hypersurface $N^n\subset M$ with negative semi-definite second fundamental form and $\mathcal{I}^{+}(N)\neq\emptyset$. Let $\psi:\Sigma^n\rightarrow M^{n+1}$ be a spacelike hypersurface such that
$\psi(\m)\subset\mathcal{I}^+(N)$, and let $v$ denote the Lorentzian distance function from $N$ along the hypersurface \m. Then
\begin{itemize}
\item[(i)] If $c>0$ and the future mean curvature of $\Sigma$ satisfies
\[
H\leq \frac{2 \sqrt{n-1}}{n}\sqrt{c}\tanh{(\sqrt{c}v)}
\]
(with $H<\sqrt{c}\tanh{(\sqrt{c}v)}$ at some point of \m\ if $n=2$), then $\Sigma$  is hyperbolic.
\item[(ii)] If $c=0$ and $H\leq 0$ (with $H<0$ at some point of \m\ if $n=2$), then $\Sigma$ is hyperbolic.
\end{itemize}
\end{theorem}


\begin{thebibliography}{AAAAA}

\bibitem{AA1} J. A. Aledo and L. J. Al\'\i as,  \textit{On the curvatures of bounded complete spacelike
hypersurfaces in the Lorentz-Minkowski space}, manuscripta math. \textbf{101} (2000), 401--413.

\medskip

\bibitem{AA2} J. A. Aledo and L. J. Al\'\i as,  \textit{On the curvatures of  complete spacelike hypersurfaces in de
Sitter Space}, Geometriae Dedicata \textbf{80} (2000), 51--58

\medskip

\bibitem{BEE} J.K. Beem, P.E. Ehrlich and K.L. Easey, \textit{Global Lorentzian Geometry}, Marcel Dekker Inc.,
New York, (1996)

\medskip

\bibitem{BEMG} J.K. Beem, P.E. Ehrlich, S. Markvorsen and G.J. Galloway,
\textit{Decomposition theorems for Lorentzian manifolds with nonpositive curvature},
J. Differential Geometry \textbf{22} (1985), 29--42.

\medskip

\bibitem{BC} G.P. Bessa and M.S. Costa, \textit{On cylindrically bounded $H$-hypersurfaces of
$\mathbb{H}^n\times\mathbb{R}$}, Differential Geometry and its Applications {\bf 26} (2008), 323--326.

\medskip

\bibitem{Bu} Y.D. Burago, V.A. Zalgaller, \textit{Geometric Inequalities}, Springer-Verlag, Berlin, (1988)

\medskip


\bibitem{CX} Q. Chen, and Y.L Xin, {\it A generalized maximum principle and its applications in geometry},
Amer. J. Math. {\bf 114} (1992), 355--366.

\medskip

\bibitem{EGK} F. Erkekoglu, E. Garc\'\i a-R\'\i o, and D. N. Kupeli,  \textit{On level sets of lorentzian distance function}, General Relativity and Gravitation \textbf{35} (2003), 1597--1615

\medskip

\bibitem{Es} J.-H. Eschenburg,  \textit{The splitting theorem for space-times with strong energy condition},
J. Differential Geometry \textbf{27} (1988), 477--491.


\medskip

\bibitem{Ga1} G.J. Galloway,
\textit{Splitting theorems for spatially closed space-times}, Comm. Math. Phys. {\bf 96} (1984), 423--429.

\medskip

\bibitem{Ga2} G.J. Galloway,
\textit{The Lorentzian splitting theorem without the completeness assumption.},
J. Differential Geometry \textbf{29} (1989), 373--387.

\medskip


\bibitem{GreW} R. Greene and H. Wu,
\textit{Function theory on manifolds which possess a pole}, Lecture Notes in
Math., vol. 699, Springer-Verlag, Berlin and New York (1979).

\medskip

\bibitem{Gri} A. Grigor'yan,  \textit{Analytic and geometric
background of recurrence and non-explosion of the Brownian motion on Riemannian
manifolds}, Bull. Amer. Math. Soc. \textbf{36} (1999), 135--249.

\medskip

\bibitem{MP1} S. Markvorsen  and V. Palmer, \textit{Transience and
capacity of minimal submanifolds}, GAFA, Geometric and Functional Analysis
\textbf{13} (2003), 915--933.

\medskip

\bibitem{MP2} S. Markvorsen  and V. Palmer, \textit{How to obtain transience from bounded radial mean curvature}, Transactions of the Amer. Math. Soc.,
\textbf{357} (2005), 3459--3479.

\medskip

\bibitem{Ne} R.P.A.C. Newman, \textit{A proof of the splitting conjecture of S.-T. Yau},
J. Differential Geom. {\bf 31} (1990), 163--184.

\medskip

\bibitem{O'N} B. O'Neill,
\textit{Semi-Riemannian Geometry; With Applications to Relativity}, Academic
Press, New York, (1983).

\medskip


\bibitem{Om} H. Omori,  \textit{Isometric immersions of Riemannian manifolds}, J. Math. Soc. Japan,
\textbf{19} (1967), 205-214.

\medskip

\bibitem{PRS} S. Pigola, M. Rigoli and A.G. Setti, \textit{Maximum principles on Riemannian manifolds and applications},
Memoirs Amer. Math. Soc. {\bf 822} (2005).

\medskip

\bibitem{Wu} B.Y. Wu, \textit{On the mean curvature of spacelike submanifolds in semi-Riemannian manifolds}, Journal of Geometry and Physics,
\textbf{56} (2006), 1728--1735.

\medskip

\bibitem{Y} S.T. Yau,  \textit{Harmonic function on complete Riemannian manifolds}, Commun. Pure Appl. Math.,
\textbf{28} (1975), 201-228.


\end{thebibliography}
\end{document}